\documentclass[11pt]{amsart}

\usepackage{amsmath}
\usepackage{amssymb}
\usepackage{graphicx}
\usepackage[dvipsnames]{xcolor}

\newtheorem{theorem}{Theorem}[section]
\newtheorem{lemma}[theorem]{Lemma}
\newtheorem{cor}[theorem]{Corollary}

\theoremstyle{definition}
\newtheorem{definition}[theorem]{Definition}
\newtheorem{example}[theorem]{Example}

\theoremstyle{remark}

\theoremstyle{problem}
\newtheorem{problem}[theorem]{Problem}

\numberwithin{equation}{section}

\newcommand{\To}{\longrightarrow}

\def\V{\Vert}

\begin{document}

\title{The Ball Fixed Point Property in  spaces of continuous functions}

\author[A.\ Avil\'es]{Antonio Avil\'es}
\address{Universidad de Murcia, Departamento de Matem\'{a}ticas, Campus de Espinardo 30100 Murcia, Spain.} 
\email{avileslo@um.es}

\author[M. Jap\'on]{Maria Jap\'on}
\address{Institute of Mathematics of the University of Sevilla, IMUS. Avda. Reina Mercedes, s/n. 41012 - Sevilla}
\email{japon@us.es}

\author[C. Lennard]{Christopher Lennard}
\address{University of Pittsburgh, Department of Mathematics, Thackeray Hall, Pittsburgh, PA 15260, United States}
\email{lennard@pitt.edu}

\author[G. Mart\'{\i}nez]{Gonzalo Mart\'{\i}nez-Cervantes}
\address{Universidad de Murcia, Departamento de Matem\'{a}ticas, Campus de Espinardo 30100 Murcia, Spain.} 
\email{gonzalo.martinez2@um.es}

\author[A. Stawski]{Adam Stawski}
\address{University of Pittsburgh, Department of Mathematics, Thackeray Hall, Pittsburgh, PA 15260, United States}
\email{ACS192@pitt.edu}

\thanks{Avil\'{e}s, Mart\'{\i}nez-Cervantes and Jap\'{o}n were supported by Fundaci\'{o}n S\'{e}neca - ACyT Regi\'{o}n de Murcia project 21955/PI/22, Avil\'{e}s and Mart\'{\i}nez Cervantes supported by Agencia Estatal de Investigación (Government of Spain) and ERDF project PID2021-122126NB-C32. Avil\'{e}s was partially supported by  the Grant PID2023-148294NB-I00 funded by MICIU/AEI/ 10.13039/501100011033 and  Avil\'{e}s and Lennard
 were  partially supported by the Institute of Mathematics of the University of Sevilla, IMUS}
\thanks{M. Japón is supported by  the Grant PID2023-148294NB-I00 funded by MICIU/AEI/ 10.13039/501100011033 and FQM-127}



\begin{abstract}

A Banach space $X$ has the ball fixed point property (BFPP) if  for every closed ball $B$ and for every nonexpansive mapping $T\colon B\to B$, there is  a fixed point. We study the BFPP for $C(K)$-spaces. Our goal is to determine topological properties over $K$ that may determine the failure or fulfillment of the BFPP for the 
space of continuous functions $C(K)$.  We prove that the class of compact spaces $K$ for which the BFPP holds  lies between  the class of extremally disconnected compact spaces and the class of compact $F$-spaces. We give a family of examples of $F$-spaces $K$ for which the BFPP fails. As a result, we prove that for every cardinal $\kappa$,  $\kappa$-order completeness  or $\kappa$-hyperconvexity of $C(K)$ are not enough for the BFPP and we obtain that $\ell_\infty/c_0 = C(\mathbb{N}^*)$  fails BFPP under the Continuum Hypothesis. The space  $C([0,+\infty)^*)$ is also analyzed. It is left as an open problem whether all compact spaces for which  the BFPP  holds are in fact the  extremally disconnected compact sets.

\end{abstract}

\maketitle

\section{Introduction}

The classical Banach and Schauder-Brouwer theorems ensure the existence of fixed points for contractions in complete metric spaces and for continuous functions in compact convex sets of Banach spaces. In the middle of these two results, the existence of fixed points for nonexpansive mappings (Lipschitz with constant $1$) is a rather delicate matter, that is far from being well understood. For instance, it is known that the existence of fixed points for a selfmap  defined on a closed convex bounded subset of Banach space is strongly related to the  geometrical features of the Banach space itself. In fact,   uniform convexity, uniform smoothness,  normal structure, Kadec Klee and Opial conditions,  and a further extensive collection of geometrical properties generally studied under scope of  Banach space theory, give support to a vast area of research trying to determine  the precise connection between the geometry of a Banach space $X$ and   the existence of fixed points for nonexpansive selfmaps. Certainly, one of the major open problems in the area is to know  whether every nonexpansive selfmap defined on a closed convex bounded subset of a reflexive Banach space (or superreflexive) has a fixed point. For a deep insight into how metric fixed point theory has evolved in the last 60 years, the reader can consult the following list of references, which is by no means exhaustive \cite{BS, todos, DJ, DJ2, ELOS, GK,KS, Lin1,Lin2}.\\
  
At this point, it is worth mentioning that, in the absence of norm-compact\-ness,  the existence of a fixed point for a nonexpansive operator is not guaranteed. In the framework of  $C(K)$ spaces, the following two examples are folklore:  
Firstly,  let $B$ denote the closed unit ball of $c$, the Banach space of all  convergent sequences or, equivalently, the space of all  continuous functions on the one-point compactification of $\mathbb{N}$. Let  $T\colon B\to B$  be defined by $T(x_1,x_2,\cdots)=(1,-1, x_1,x_2,\cdots )$. Then $T$  is   a fixed point free isometry. Secondly, let $B$ now denote  the closed unit ball of $C([-1,1])$ and  $T\colon B\to B$ be given by $T(f)(t)=\min\{1, \max\{-1, f(t)+2t\}\}$.  Likewise, it can be proved that $T$ is nonexpansive and fixed point free \cite[Example 3.4]{GK}. \\


We introduce some basic notations and definitions.

A nonexpansive mapping over a  metric space $(M,d)$ is a function $T\colon M\To M$ such that $d(Tx,Ty)\leq d(x,y)$ for all $x,y\in M$.   It is said that the metric space $(M,d)$ has the fixed point property (FPP) if every nonexpansive mapping $T\colon M\to M$ has a fixed point. Given a Banach space $X$, we denote its closed unit ball as $B_X = \{x\in X : \|x\|\leq 1\}$. We introduce the following definition:
\begin{definition}
	A Banach space $X$ has the ball fixed point property (BFPP) if every nonexpansive mapping $T\colon B_X\to B_X$,  has a fixed point. 
\end{definition}

Notice that the above definition is equivalent to requiring that every closed ball has the FPP, since the problem can be transferred  to  the existence of a fixed point for a nonexpansive map defined on the closed unit ball  of $X$ by composing the original mapping with a translation and a dilation.

As it is customary, all topological spaces will be assumed to be Hausdorff, and given a compact space $K$, we consider the Banach space $$C(K) = \{f\colon K\To\mathbb{R} : f\text{ is continuous}\}$$
endowed with the norm $\|f\|_\infty = \max\{|f(x)| : x\in K\}$.  \\

Motivated by the previous counterexamples in $C(\mathbb{N}\cup\{\infty\})$ and $C([-1,1])$ (which fail the BFPP) and by the fact that $\ell_\infty=C(\beta\mathbb{N})$ does fulfill the BFPP (see below), the following question seems natural and it is the main object of this paper:

\begin{problem}\label{mainproblem}
For what compact spaces $K$ does $C(K)$ have the ball fixed point property?	
\end{problem}


Through the article, we will prove that this family of compact sets lies between the class of extremally disconnected compact spaces and the class of compact F-spaces. Namely, the following will be shown:
$$K \text{ is extremally disconnected} \Longrightarrow C(K) \text{ has BFPP} \Longrightarrow   K \text{ is an  }F\text{-space.}$$

The first implication is proved in Section~\ref{ed} as a consequence of a result of Baillon \cite{Baillon} about hyperconvex metric spaces. It generalizes some existing results in the literature. Sine \cite{Sine} showed that spaces $L^\infty(\mu)$ for finite measure $\mu$ enjoy this property, and these are isometric to spaces $C(K)$ where $K$ is the Stone space of the corresponding measure algebra.  More generally, in the context of Banach lattices, Soardi \cite{Soardi} proved that duals of $AL$-spaces have the BFPP, and that includes $C(K)$ for $K$ so-called hyperstonian. A particular case of this is $\ell_1^* = \ell_\infty = C(\beta\mathbb{N})$. All these compact spaces are extremally disconnected. Actually, we do not know whether there is a non-extremally-disconnected compact $K$ for which  the BFPP in $C(K)$ still holds. (Notice that $\ell_\infty$ contains isometrically every separable Banach space so it seems useless to study the FPP in $\ell_\infty$ beyond the family of closed balls, since every counterexample of a fixed point free nonexpansive selfmap defined on a separable domain finds its replica in $\ell_\infty$).\\

The second implication of the previous scheme  is proved in Section~\ref{F}, by a straightforward exhibition of a fixed point free nonexpansive mapping from the failure of the $F$-space condition. Nevertheless, we show that being an $F$-space is not sufficient to ensure the BFPP.  In Section~\ref{sectionExamples} we  exhibit a category of  compact $F$-spaces $K$ for which every $C(K)$ fails the BFPP, where $K$ depends on two infinite cardinals $\kappa<\Gamma$. What is more, extremal disconnection is equivalent to  order-completeness of $C(K)$ in the pointwise order, and our examples show that neither order-completeness relative to a cardinality nor hyperconvexity relative to a cardinality are enough for the BFPP. \\


Once that we know that $C(K)$ fails the BFPP when $K$ is  non-$F$-space (Section 3) but that being  an $F$-space seems not to be determinant for the fulfilment of the BFPP (Section 4), in Section~\ref{remainder}, we will focus on a very particular   $F$-space: the remainder  $\mathbb{N}^*=\beta\mathbb{N}\setminus\mathbb{N}$.  Under the Continuum Hypothesis, we can derive that the space of continuous functions $C(\mathbb{N}^*)$, which is isometric to the classical quotient space $\ell_\infty/c_0$, fails the BFPP. We do not know whether CH is really necessary. A similar statement for the remainder of $\beta[0,+\infty)$ is also exhibited.\\

We will conclude  the paper with some comments and open problems that are derived from the analysis developed in the previous sections.

\section{Extremal Disconnection:  A sufficient condition for $C(K)$ to have   the BFPP}\label{ed}

The notion of hyperconvexity was introduced by Aronszajn and Panichpakdi in \cite{AP} as  an intrinsic  tool for extending  uniform continuous mappings.

\begin{definition}
	A metric space $(M,d)$ is said to be hyperconvex if for every collection of points $(x_\alpha)_{\alpha\in I}\subset M$ and positive radii $\{r_\alpha\}_{\alpha\in I}$ with the property that $d(x_\alpha, x_\beta)\le r_\alpha+r_\beta$ for all $\alpha,\beta\in I$, we have a nonempty intersection of closed balls $$\bigcap_{\alpha\in I} \{ x\in M ; d(x,x_\alpha)\leq r_\alpha \}\ne\emptyset.$$  	
\end{definition}

Notice that if $M$ is a convex subset of a normed space,  condition  $d(x_\alpha, x_\beta)\le r_\alpha+r_\beta$ is equivalent to $B(x_\alpha,r_\alpha)\cap B(x_\beta,r_\beta) \neq \emptyset$ for $\alpha\ne\beta$ (here $B(x,r)$ denotes the closed ball of center $x$ and radius $r$), so hyperconvexity means that any family of mutually intersecting closed balls has nonempty intersection. This is also called the ball intersection property, binary intersection property or similar. Besides, notice  that a normed space $X$ is hyperconvex if and only if a closed ball in $X$ is hyperconvex.

\begin{theorem}[Baillon \cite{Baillon}] \label{hyper}
	If $(M,d)$ is a bounded hyperconvex metric space, then every nonexpansive $T\colon M\to M$ has a fixed point. 
\end{theorem}

From this, $C(K)$ has the BFPP whenever $C(K)$ is hyperconvex. And when is $C(K)$ hyperconvex? The answer is in the following theorem that can by found in \cite[Theorem 2.1]{Z}, \cite[Theorems 4.3.6 and 4.3.7]{AK}

\begin{theorem}[Kelley \cite{K}, Goodner \cite{Goodner}, Nachbin \cite{Nachbin}]\label{Zippin}
	For a Banach space $X$ the following assertions are equivalent: 
		\begin{enumerate}
		
		\item $X$ is hyperconvex.
		
		\item $X$ is isometrically injective (also called 1-injective, or $P_1$-space)
		
		\item $X$ is linearly isometric to some $C(K)$ with $K$ extremally disconnected. 
		
		\item\label{order-complete} $X$ is linearly isometric to some $C(K)$ which is order-complete.
		\end{enumerate}
\end{theorem} 

\begin{cor}\label{extdisc}
	$C(K)$ has the BFPP whenever $K$ is extremally disconnected.
\end{cor}

We recall the definitions. We say that $C(K)$ is order-complete if every subset of $C(K)$ with an upper bound has a supremum in $C(K)$, where the upper bound and the supremum are considered with respect to the pointwise partial order, $f\leq g$ when $f(x)\leq g(x)$ for all $x\in K$. This is equivalent, as stated above, to $K$ being an extremally disconnected compact space, which means that the closure of every open subset of $K$ is also open (see also \cite[Proposition 7.7]{Schaefer}). We will not need this here, but we also recall that a Banach space $X$ is isometrically injective if whenever $Y\subset Z$ are Banach spaces, every bounded linear operator $T\colon Y\To X$ extends to a linear operator $\tilde{T}\colon Z\To X$ of the same norm.
\medskip

The hyperstonian spaces, that include the Stone spaces of measure algebras mentioned at the beginning of this section, are all extremally disconnected, but they have the extra property that $C(K)$ can be represented as a dual space. There are however extremally disconnected spaces $K$ for which $C(K)$ is not even isomorphic to any dual space (see for instance \cite[Section 1.4]{A} or \cite[Theorem 4.3]{LC}), so duality does not play any key role in the BFPP  in contrast to the arguments used by Soardi \cite{Soardi}.
\medskip 

For a more general picture of the class of extremally disconnected spaces and their spaces of continuous functions we refer to \cite{Walker}. We may mention just a few facts.  Extremally disconnected compact spaces are exactly the Stone spaces of order-complete Boolean algebras. Also, for every compact space $L$ there is a unique extremally disconnected compact space $K$ and continuous surjection $\phi\colon K\To L$ (this is called the Gleason cover of $L$) such that $\phi$ does not map any proper closed subset of $K$ onto $L$. In this situation, there is a natural inclusion $C(L)\hookrightarrow C(K)$ and $C(K)$ is the natural order-completion of $C(L)$ \cite{Gleason}. Of particular interest is the Gleason space of a convergent sequence, which is the same as the \v{C}ech-Stone compactification $\beta\mathbb{N}$ of the natural numbers. The compact space $\beta\mathbb{N}$ contains a dense discrete open copy of $\mathbb{N}$ and every bounded function $\mathbb{N}\To \mathbb{R}$ extends to a unique continuous function $\beta\mathbb{N}\To \mathbb{R}$, making $C(\beta\mathbb{N})$ isometric to the space of bounded sequences $\ell_\infty$. The compact space $\mathbb{N}^* = \beta\mathbb{N}\setminus\mathbb{N}$, known as the \v{C}ech-Stone remainder of the natural numbers, is not extremally disconnected.

\medskip

\section{ $F$-spaces:  a necessary condition for $C(K)$ to have the BFPP}\label{F}

The concept of $F$-space was introduced in \cite{GH} 
  as a topological space $K$ for which every finitely generated ideal in $C(K)$ is principal. Equivalent definitions are the following:
  
  \begin{lemma}\label{equiv-F}\cite[Section 14.25]{GH} Let $K$ be a compact  space. The following are equivalent:
  
  \begin{itemize} 
  
  \item[i)] $K$ is an $F$-space.
  \item[ii)] Disjoint cozero sets have disjoint closures.
  \item[iii)] For every continuous function $g\in C(K)$ there exists a continuous $f\colon K\To [-1,1]$ such that $g\cdot f = |g|$. 
  
  \end{itemize}
  \end{lemma}
  
Notice that the function  $f$ in Lemma \ref{equiv-F}.iii)  can be understood as a \emph{continuous sign} for the given $g$,  since $f$ takes value $1$ where $g$ is positive, $-1$ where $g$ is negative and it freely  transitions  from positive to negative in a continuous way where $g$ vanishes.   

Every extremally disconnected compact space is clearly an $F$-space. It turns out that being an $F$-space is a necessary condition for the BFPP, as the following theorem shows (note that the proof is related to that of Theorem 2 of \cite{ACM}):

  \begin{theorem}\label{non-F} Let $K$ be a compact space. If   $C(K)$ verifies the BFPP, then $K$ is an $F$-space.
  \end{theorem}

  \begin{proof} 
Take $g\in C(K)$ and we will find $f$ such that $g\cdot f = |g|$. Without loss of generality we can consider that $g$ is non-null and by considering ${g}/{\|g\|_\infty}$ we can suppose  that $\|g\|_\infty\leq 1$. We define a mapping $T\colon C(K)\To C(K)$ given by
  	$$
  	Tf =\left(1-|g|\right)f + g.
  	 $$
A continuous function $f$ verifying $T(f)=f$  automatically satisfies $g\cdot f = |g|$ so   
we are looking for a fixed point of $T$. Notice that $T(B_{C(K)})\subseteq B_{C(K)}$ because $\|g\|_\infty\leq 1$, and so, for every $f\in B_{C(K)}$ and $x\in K$,
  $$|Tf(x)| \leq \left(1-|g(x)|\right)|f(x)| + |g(x)| \leq 1-|g(x)| + |g(x)| = 1.$$
   Since we assume the BFPP we only need to check that $T$ is nonexpansive.  Indeed, for any $f,f'\in C(K)$ and for any $x\in K$,
  $$|Tf - Tf'|(x) = ( 1 - |g(x)|) |f(x)-f'(x)| \leq |f-f'|(x).$$
  The above shows that $K$ is an $F$-space.
  \end{proof}
  
In order to understand the consequences of this fact, we need to have a picture of the class of compact $F$-spaces. We can recall, for instance, that a compact $F$-space cannot contain any nontrivial convergent sequence. This is because, if $\{x_n\}$ converges to $x\not\in\{x_n: n\in\mathbb{N}\}$, then by Tietze's extension theorem, there is a continuous $g$ with $g(x_n) = (-1)^n/n$, and no continuous $f$ can satisfy $g = f|g|$.

\begin{cor}\label{sequence} If $K$ contains a nontrivial convergent sequence, then $C(K)$ fails the BFPP. In particular, $C(K)$ fails the BFPP for metrizable $K$.
\end{cor}

The product of two infinite compact spaces always fails to be an $F$-space as well \cite[11Q.2, page 215]{GJ}. This is because, as it can be deduced from complete regularity, for $i=1,2$, we could find points $x_i^n\in K_i$ and continuous functions $g_i\colon K_i\To [0,1]$ such that $g_i(x^n_i)=1/n$. For $g(p,q) = g_1(p)-g_2(q)$ we cannot find $f$ with $g=f|g|$. If $x_i$ is a cluster point of the sequence $\{x^n_i\}$ for $i=1,2$, then we should have $f(x_1,x_2)=1$ because $(x_1,x_2)\in \overline{\{(x_1^n,x_2^m) : n<m\}}$, but also $f(x_1,x_2)=-1$ because $(x_1,x_2)\in \overline{\{(x_1^n,x_2^m) : n>m\}}$.

\begin{cor}\label{product}
	$C(K_1\times K_2)$ fails the BFPP for any infinite compact spaces $K_1$, $K_2$.
\end{cor}

As a consequence, the space $C(\beta \mathbb{N}\times \beta\mathbb{N})$ fails the BFPP in contrast to  $C(\beta\mathbb{N})$. Since $C(K_1\times K_2)$ is isometric to $C(K_1, C(K_2))$, the space of all continuous functions from $K_1$ to $C(K_2)$, we deduce that $C(\beta\mathbb{N}, \ell_\infty)$ is not an injective Banach space or, equivalently, it fails the hyperconvexity property as a metric space. \\

\begin{example}

If $p_1\neq p_2$ are points of $\beta\mathbb{N}\setminus \mathbb{N}$, the closed  subspace $$
	X=\{g\in C(\beta\mathbb{N}): g(p_1)=g(p_2)\}
 $$  fails to have the BFPP.

 We can deduce this from Theorem~\ref{non-F} and \cite[Example 6.9.8]{Lau}. As $X$ is isometric to $C(K)$ where 
	$K$ is the quotient space of $\beta\mathbb{N}$ obtained by identifying the points $p_1$ and $p_2$, it is sufficient to notice that $K$ is not an $F$-space.  Let $A_1$, $A_2$ be two disjoint clopen subsets of $\beta\mathbb{N}$ such that $p_i\in A_i$. We can define a continuous function $g\in C(\beta\mathbb{N})$ such that $g(n)=(-1)^i /n$ for $n\in A_i\cap \mathbb{N}$. This can be viewed as a continuous function $\hat{g}\in C(K)$ because $g(p_1)=g(p_2)=0$. However, if a continuous $f$ satisfies $g = f|g|$, then $f(p_1)=-1$ and $f(p_2)=1$.

\end{example}

As we have mentioned, every extremally disconnected compact space $K$ is an $F$-space. This is because if $g\in C(K)$, then $\overline{\{x : g(x)<0\}}$ and 
 $\overline{\{x : g(x)>0\}}$ are disjoint clopen sets, so one can define $f$ that takes value $-1$ on the first set, $1$ on the second set and $0$ elsewhere. One key difference, however, between extremally disconnected and $F$-spaces is that, by virtue of Tietze's extension theorem, a closed subspace of a compact $F$-space is also an $F$-space. The \v{C}ech-Stone remainder $\mathbb{N}^* = \beta\mathbb{N}\setminus \mathbb{N}$ is the most natural example of a compact $F$-space that is not extremally disconnected, and so it is not covered  by the necessary nor by the sufficient condition for BFPP that we have so far.

\section{Compact $F$-spaces for which $C(K)$ fails the  BFPP}\label{sectionExamples}

In this section we will show that being a compact $F$-space is not a sufficient condition for the BFPP. In order to do that we will introduce a family of Banach spaces, initially formed by bounded functions and depending   on two infinite cardinals, which all fail the BFPP. Later on, we will prove that these Banach spaces can isometrically be  identified with $C(K)$-spaces with $K$ being an $F$-space. \\

As it is usual in set theory, cardinals are special kind of ordinals and an ordinal $\alpha$ is the same as the set of all ordinals below $\alpha$.

Let $\kappa<\Gamma$ be two infinite cardinals. We denote by $\ell_\infty(\Gamma)$ the Banach space of all bounded functions $f\colon \Gamma\To \mathbb{R}$ endowed with the supremum norm $\|f\|_\infty = \sup_{\gamma\in \Gamma}|f(\gamma)|$, and by $[\Gamma]^\kappa$ the family of all subsets of $\Gamma$ of cardinality $\kappa$. Define
\begin{eqnarray*} X^{\kappa,\Gamma} &:=& \{f\in\ell_\infty(\Gamma) : \exists A\in [\Gamma]^\kappa \ \ f|_{\Gamma\setminus A} \text{ is constant.}\}
  \end{eqnarray*} 
 $X^{\kappa,\Gamma}$  is a  linear subspace of $\ell_\infty(\Gamma)$ which is also closed: If $f_n$ is constant equal to $c_n$ on $\Gamma\setminus A_n$, and $\lim \|f_n - f\|_\infty=0$, then taking $A=\bigcup_n A_n$, each $f_n$ is constant equal to $c_n$ on $\Gamma\setminus A$, so $f$ is constant equal to $\lim_n c_n$ on $\Gamma\setminus A$. So $X^{\kappa,\Gamma}$ is a Banach space. We next show that:

 \begin{theorem}\label{XnotBFPP}
	The space $X^{\kappa,\Gamma}$ fails the BFPP.
\end{theorem}

\begin{proof}
	 For $f\in \ell_\infty(\Gamma)$, we define $Tf\in \ell_\infty(\Gamma) $  as follows:
	 \begin{itemize}
	 \item  $Tf(0) = 1$, $Tf(1)=-1$.
	 \item For an ordinal of the form $\alpha+2$ , $Tf(\alpha+2) = f(\alpha)$.
	 \item For a limit ordinal $\beta$, \begin{eqnarray*}Tf(\beta) &=& \inf_{\alpha<\beta}\sup_{\alpha<\gamma<\beta} f(\gamma)
	 \end{eqnarray*}
 \item For a successor of a limit ordinal $\beta+1$,
 $$	Tf(\beta+1) = \sup_{\alpha<\beta}\inf_{\alpha<\gamma<\beta} f(\gamma)$$
	 \end{itemize} 
 
Notice that  $Tf$ is a shift forward where  $1$ and $-1$ are placed in the first two spots and
the values of $f$ have been moved two positions to the right; the rest of pairs of vacant places are filled with  either the limit superior or inferior of the preceding values. \\

The first remark is that $T(X^{\kappa,\Gamma})\subseteq X^{\kappa,\Gamma}$. To see this,
given a set $A\subset \Gamma$, let $\overline{A} = \{\sup(B) : B\subseteq A\}$. If $f$ is constant equal to $c$ outside of $A$, then $Tf$ is constant equal to $c$ out of $\{\alpha+1,\alpha+2: \alpha\in A\}\cup \{\beta,\beta+1 : \beta\in \overline{A}\}$, so we only need to notice that $|\overline{A}|\leq |A|$. This is because there is an injective mapping $h:\overline{A}\setminus \{\sup(\overline{A})\}\To A$ given by $h(\gamma) = \min\{\alpha\in A : \alpha >\gamma\}$.

 Moreover, $T(B_{X^{\kappa,\Gamma}})\subseteq B_{X^{\kappa,\Gamma}}$ because if $f$ takes values in $[-1,1]$, so will  $Tf$.

 The third fact that we need is that $T\colon B_{X^{\kappa,\Gamma}}\To B_{X^{\kappa,\Gamma}}$ is nonexpansive,  that is, that $|Tf(\alpha)-Tf'(\alpha)|\leq \|f-f'\|_\infty$ for any $f,f'\in  B_{X^{\kappa,\Gamma}}$ and $\alpha\in \Gamma$. Only the case of a limit  ordinal $\beta$ and its successor $\beta+1$ needs to be checked. Remember that if $***$ stands for either $\inf$ or $\sup$, and $\{r_i : i\in I\}$ and $\{s_i :i\in I\}$ are families of real numbers, then 
$$\left|\underset{i\in I}{***}\ r_i - \underset{i\in I}{***}\ s_i \right| \leq \sup_{i\in I}|s_i-r_i|, $$
so \begin{eqnarray*}|Tf(\beta) - Tf'(\beta)|& = & \left|\inf_{\alpha<\beta}\sup_{\alpha<\gamma<\beta} f(\gamma) - \inf_{\alpha<\beta}\sup_{\alpha<\gamma<\beta} f'(\gamma)\right|\\ &\leq& 
\sup_{\alpha<\beta} \left|\sup_{\alpha<\gamma<\beta} f(\gamma) - \sup_{\alpha<\gamma<\beta} f'(\gamma)\right|\\ &\leq& 
\sup_{\alpha<\beta} \sup_{\alpha<\gamma<\beta} \left| f(\gamma) - f'(\gamma)\right| \leq \|f-f'\|_\infty.
\end{eqnarray*}
Analogously, $|Tf(\beta+1) - Tf'(\beta+1)|\le \V f-f'\V_\infty$.\\

 Finally assume that we have a fixed point $Tf = f\in B_{X^{\kappa,\Gamma}}$. Consider the set of ordinals $$M = \{\text{ limit ordinals } \beta<\Gamma: (f(\beta), f(\beta+1)) \neq (1,-1)\}.$$ Since $f$ is constant out of a set of ordinals of cardinality $\kappa$, the set $M$ is nonempty, so $M$ has a minimum $\mu = \min(M)$. By minimality $f(\beta)=1$ and $f(\beta+1)=-1$ for all limit ordinals $\beta<\mu$. For $\beta=0$, remember that $f(0)=Tf(0)=1$ and $f(1) = Tf(1)=-1$, so $0\not\in M$. Moreover, by the formula $f(\alpha+2) = Tf(\alpha+2) = f(\alpha)$ we have that $f(\beta+2n)=1$ and $f(\beta+2n+1)=-1$ for every limit ordinal $\beta<\mu$ and every natural number $n$. But then,
 \begin{eqnarray*} f(\mu) = Tf(\mu) &=& \inf_{\alpha<\mu}\sup_{\alpha<\gamma<\mu} f(\gamma) = 1 \\
 	f(\mu+1) = Tf(\mu+1) &=& \sup_{\alpha<\mu}\inf_{\alpha<\gamma<\mu} f(\gamma) = -1
 \end{eqnarray*}
This means that $\mu\not\in M$, a contradiction.
\end{proof}

 \begin{lemma}
For  $\kappa<\Gamma$ two infinite cardinals, the Banach space $X^{\kappa,\Gamma}$ is isometric to a space $C(K^{\kappa,\Gamma})$, where $K^{\kappa,\Gamma}$ is a compact $F$-space. 
 \end{lemma}

\begin{proof}
Notice that $X^{\kappa,\Gamma}$ contains the constant functions and is closed under multiplication, so we invoke Arens' characterization of real Banach algebras of continuous functions \cite[p.~281]{Arens}, and we conclude that there is a linear onto isometry that also preserves pointwise multiplication between $X^{\kappa,\Gamma}$ and $C(K^{\kappa,\Gamma})$ for some compact space $K^{\kappa,\Gamma}$. This isometry also preserves pointwise order because $f\leq g$ if and only if $g-f=h^2$ for some $h$, and therefore it also preserves the absolute value $|f| = \max(-f,f)$. The compact space $K^{\kappa,\Gamma}$ is an $F$-space because if $g\in X^{\kappa,\Gamma}$ and we define $f\in \ell_\infty(\Gamma)$ by $$f(\alpha) = \begin{cases}
1 &\text{ if } g(\alpha)>0,\\
0 &\text{ if } g(\alpha)=0,\\
-1 &\text{ if } g(\alpha)<0,\\
\end{cases}$$
then $f\in X^{\kappa,\Gamma}$ and $g=f\cdot |g|$.\\

\end{proof}

We do not enter into details here, but a standard argument with Boolean algebras shows that $K^{\kappa,\Gamma}$ is the Stone space of the Boolean algebra of subsets $A$ of $\Gamma$ such that $\min(|A|,|\Gamma\setminus A|)\leq \kappa$. 
Another description is that $K^{\kappa,\Gamma}$ is the result of glueing into a single point all ultrafilters of $\beta\Gamma$ that do not contain any subset of cardinality at most $\kappa$. \\

At this point, what we know is that the class of compact spaces $K$ for which $C(K)$ enjoys the BFPP is strictly contained in the class of compact $F$-spaces and contains at least all extremally disconnected spaces.  There is a characterization of $F$-spaces in the spirit similar to that in Theorem~\ref{Zippin} that $K$ is extremally disconnected if and only if $C(K)$ is order-complete: in fact,  $K$ is an $F$-space if and only if whenever two countable sets $A,B\subset C(K)$ satisfy $a\leq b$ for all $a\in A$, $b\in B$ there exists $f\in C(K)$ such that $a\leq f\leq b$ for all $a\in A$, $b\in B$, see for instance \cite[Theorem 5.2]{aleph}  or  \cite[Theorem 2.3.4]{Lau}.
The fact that extremally disconnected spaces are $F$-spaces can be seen in this way because if the supremum $f=\sup(A)$ exists, it would do the job. It is natural to wonder whether some completeness property weaker than full order-completeness but stronger than the countable separation property of $F$-spaces could still be sufficient for BFPP. However, our example shows the following:

\begin{theorem}
	For every infinite cardinals $\kappa<\Gamma$, $C(K^{\kappa,\Gamma})$ is $\kappa$-complete, in the sense that every bounded subset of cardinality at most $\kappa$ has a supremum in $C(K^{\kappa,\Gamma})$. Yet, $C(K^{\kappa,\Gamma})$ fails the BFPP.
\end{theorem}

\begin{proof}
We already checked that $X^{\kappa,\Gamma}  \cong C(K^{\kappa,\Gamma})$ fails BFPP in Theorem~\ref{XnotBFPP}. On the other hand, $X^{\kappa,\Gamma}$ is clearly $\kappa$-complete. Namely, if $f_i$ is constant equal to $c_i$ out of a set $A_i$ , then the pointwise supremum $\sup_{i\in I} f_i$ is constant equal to $\sup_{i\in I} c_i$ out of $\bigcup_{i\in I}A_i$. And the union of at most $\kappa$ many sets of cardinality $\kappa$ has cardinality $\kappa$.
\end{proof}

The property that $C(K)$ is $\kappa$-complete is equivalent to saying that $\overline{W}$ is open whenever $W$ is an open set that is the union of at most $\kappa$ many closed sets. When $\kappa=\aleph_0$, these compact spaces are called basically disconnected. In a similar way, hyperconvexity in Theorem \ref{hyper} cannot be relaxed by a cardinality restriction, because if $C(K)$ is $\kappa$-complete, then the intersection of at most $\kappa$ many mutually intersecting closed balls is nonempty. This is because a closed ball of center $f$ and radius $r$ in $C(K)$ can be viewed as the interval $[f-r,f+r]$, so a nonempty intersection $[f-r,f+r]\cap [g-s,g+s] \neq \emptyset$ translates into $f-r\leq g+s$ and $g-s\leq f+r$, and if $\{[f_i-r_i,f_i+r_i]\}$ are mutually intersecting balls, $\sup_{i\in I}\{f_i-r_i\}$ belongs to the intersection.\\

The reader may have noticed that the construction of the fixed point free mapping given in Theorem \ref{XnotBFPP} can be seen as a transfinite modification of the classical example of  the two-position shift   that shows that the separable Banach space $c= C(\mathbb{N}\cup\{\infty\})$ fails the BFPP. 
Another classical example  within the separable scope is $c_0$, the Banach space of all null convergent sequences or, equivalently, all continuous functions in $C(\mathbb{N}\cup\{\infty\})$ vanishing at infinity. In this case, the single shift $T\colon B_{c_0}\to B_{c_0}$ given by $T(x_1,x_2,\cdots)=(1,x_1,x_2,\cdots)$ is a fixed point free isometry.\\
A transfinite modification of this example can be given by the mapping $S\colon X^{\kappa,\Gamma} \longrightarrow X^{\kappa,\Gamma}$ defined as follows:
\begin{itemize}
\item  $Sf(0) = 1$,
\item For a successor ordinal $\alpha+1$, $Sf(\alpha+1) = f(\alpha)$,
\item For a limit ordinal $\beta$, $$Sf(\beta) = \inf_{\alpha<\beta}\sup_{\alpha<\gamma<\beta} f(\gamma).$$ 		
\end{itemize}
Let $X_0^{\kappa,\Gamma}$ denote the subspace of $X^{\kappa,\Gamma}$ consisting of those functions vanishing outside a set of cardinality $\kappa$, i.e.~  
\begin{eqnarray*} X_0^{\kappa,\Gamma} &:=& \{f\in\ell_\infty(\Gamma) : \exists A\in [\Gamma]^\kappa \ \ f|_{\Gamma\setminus A} \text{ is null}\}.
\end{eqnarray*} 

The argument in the proof of Theorem \ref{XnotBFPP} shows that $S$ is a fixed point free mapping on $X_0^{\kappa,\Gamma}$. Nevertheless, we do not repeat the details here since the mapping $T$ constructed in the proof of Theorem \ref{XnotBFPP} trivially satisfies $TX_0^{\kappa,\Gamma} \subseteq X_0^{\kappa,\Gamma}$. Therefore, in any case we obtain the following conclusion:

\begin{theorem}\label{X0notBFPP}
	The space $X_0^{\kappa,\Gamma}$  fails the BFPP.
\end{theorem}

Notice that the nonexpansive mappings without fixed points that we have presented in this section are actually isometries. \medskip

\bigskip

Let us finish this section with an observation about the space $X_0^{\kappa,\Gamma}$.
By Arens' theorem we know that there is a multiplication and order preserving linear isometry $f\mapsto \hat{f}$ from $X^{\kappa,\Gamma}$ to $C(K^{\kappa,\Gamma})$. Consider the following family of closed subsets of $K^{\kappa,\Gamma}$: $\left\{\hat{f}^{-1}(0) : f\in X^{\kappa,\Gamma}_0\right\}$. This family has the finite intersection property because $c\cdot 1 \not\leq f^2_1+\cdots+f^2_n$ for any $c\in \mathbb{R}^+$, $f_1,\ldots,f_n\in X^{\kappa,\Gamma}_0$. By compactness, there is a point $p^{\kappa,\Gamma}\in K^{\kappa,\Gamma}$ where $\hat{f}$ vanishes for all $f\in  X^{\kappa,\Gamma}_0$.  Define $C(K^{\kappa, \Gamma},p^{\kappa,\Gamma})=\{g\in C(K^{\kappa, \Gamma}): g(p^{\kappa,\Gamma})=0\}$.

The above means that $$\{\hat{f} : f\in X^{\kappa,\Gamma}_0\} \subseteq C(K^{\kappa, \Gamma},p^{\kappa,\Gamma}),$$
 but since both are subspaces of codimension 1  in $C(K^{\kappa,\Gamma})$, we conclude that
 $$\{\hat{f} : f\in X^{\kappa,\Gamma}_0\} = C(K^{\kappa,\Gamma},p^{\kappa,\Gamma})$$
 and $X^{\kappa,\Gamma}_0$ is isometric to  $C(K^{\kappa,\Gamma},p^{\kappa,\Gamma})$.\\

\section{Under CH, $C(\mathbb{N}^*)$ fails the BFPP}\label{remainder}

We now have a family of examples of compact $F$-spaces  for which the BFPP fails. But what about \emph{the most prominent example} of compact $F$-space that is not extremally disconnected, that is, $\mathbb{N}^* = \beta\mathbb{N}\setminus\mathbb{N}$, the \v{C}ech-Stone remainder of the natural numbers. We show now that $C(\mathbb{N}^*)$ fails the BFPP under the Continuum Hypothesis (CH). We do not know if this is true in ZFC or if there are models of set theory where $C(\mathbb{N}^*)$ has the BFPP. 

\medskip

Remember that $C(\mathbb{N}^*)$ is isometric to $\ell_\infty/c_0$. The natural approach would be to exhibit an explicit nonexpansive mapping on the closed ball of $\ell_\infty/c_0$ into itself without fixed points. Nevertheless,  we have not been able to do so. Instead, we rely on the general structural properties that $\mathbb{N}^*$ has under CH, that constitute what van Mill calls its \emph{smiling friendly head} \cite{topo}, to make one of our examples from Section~\ref{sectionExamples} sit in $\mathbb{N}^*$.

\medskip

\begin{lemma}\label{retractBanach}
	Suppose that we have two nonexpansive linear operators between Banach spaces $E\colon X\To Y$ and $R\colon Y\To X$ such that $E\circ R = Id_Y$. If $X$ has the BFPP, then $Y$ has the BFPP.
\end{lemma}

\begin{proof}
	Let $T\colon B_{Y}\to B_{Y}$ be nonexpansive. We can create a new nonexpansive mapping
	$$\tilde{T} = R \circ T \circ E\colon B_{X}\To B_{X}$$
	that must have a fixed point $g\in B_X$. Then $f = Eg$ is the desired fixed point of $T$, because
	$$Tf = T(Eg) = ERT(Eg) = E\tilde{T}g = Eg = f.$$
\end{proof}

\begin{definition}\label{retractdef}
Let $K,L$ be compact topological spaces. We say that $L$ is a continuous retract of $K$ if there exist continuous functions $e\colon L\to K$ and $r\colon K\to L$ such that $r\circ e=Id_L$.
\end{definition}

\begin{lemma}\label{BFPPretract}
If $L$ is a continuous retract of $K$ and $C(K)$ has the BFPP, then $C(L)$ has the BFPP as well.

\end{lemma} 
\begin{proof}
Remember that every continuous function between compact spaces $\phi:K_1\To K_2$ induces a nonexpansive linear operator $\phi^o: C(K_2)\To C(K_1)$ acting by composition $\phi^o(f) = f\circ \phi$. If $e$ and $r$ are as in Definition~\ref{retractdef} we can apply Lemma~\ref{retractBanach} to $E=e^o$ and $R=r^o$.
\end{proof} 

The key fact is the following result in general topology, cf. Theorem 1.4.4 and Theorem 1.8.1 in \cite{topo}.

\begin{theorem}\label{retractN}
	Under CH, every zero-dimensional compact $F$-space of weight $\omega_1$ is a continuous retract of $\mathbb{N}^*$.
\end{theorem}

The weight of a topological space is, by definition, the least cardinality of a basis of open sets. What we will use is that the weight of an infinite compact space $K$ equals the least cardinality of a dense subset of the Banach space $C(K)$, cf. \cite[Proposition 7.6.5]{S}. A compact space $K$ is called zero-dimensional if for every two different points $x,y$ there is a clopen set $A$ such that $x\in A$, $y\not\in A$. In this case, in fact, for every two disjoint closed sets $F,G$ there is a clopen set $A$ with $F\subseteq A$ and $G\subseteq K\setminus A$ (see  \cite[Proposition 8.2.2]{S}). We can translate this fact into the property that for every $f\in C(K)$ there exists $g\in C(K)$ with $g^2=1$ and such that $f\cdot g>-1$ (here $g$ gives a clopen set separating the closed sets $\{f\geq 1\}$ and $\{f\leq -1\}$). This property holds in $X^{\kappa,\Gamma}$. Indeed, given $f\in X^{\kappa,\Gamma}$, we can just take $g$ to take value 1 where $f$ is nonnegative and value $-1$ where $f$ is negative.  
Thus, every  compact set $K^{\kappa,\Gamma}$ is zero-dimensional. 

\begin{theorem}\label{CH}
	Under CH, $C(\mathbb{N}^*) = \ell_\infty/c_0$ fails the BFPP.
\end{theorem}

\begin{proof}
We check that  the compact space $K^{\omega,\omega_1}$ from Section~\ref{sectionExamples} satisfies the remaining hypotheses of Theorem \ref{retractN}, that is, that  the weight of $K^{\omega,\omega_1}$ is
 $\omega_1$ (under CH). 
  As for the computation of the weight, since each function of $X^{\omega,\omega_1}\cong C(K^{\omega,\omega_1})$ is constant out of a countable set, we can define an injective mapping
$$X^{\omega,\omega_1} \To\mathbb{R} \times \omega_1^\mathbb{N} \times \mathbb{R}^\mathbb{N}$$ that sends each $f\in X^{\omega,\omega_1}$ to a triple $(s,(\gamma_n),(r_n))$ so that $f(\gamma_n)=r_n$ and $f(\beta) = s$ whenever $\beta\not\in \{\gamma_n : n\in\mathbb{N}\}$. Thus, the cardinality of $C(K^{\omega,\omega_1})$ is, under CH, at most
$$|\mathbb{R}| \cdot |\omega_1^\mathbb{N}| \cdot |\mathbb{R}^\mathbb{N}| = |\mathbb{R}| \cdot |\mathbb{R}^{\mathbb{N}}| \cdot |\mathbb{R}^\mathbb{N}| = |\mathbb{R}| = \omega_1.$$
The weight of $K^{\omega,\omega_1}$ is therefore at most $\omega_1$. It cannot be less than that, because the characteristic functions of singletons are $\omega_1$ many elements in $X^{\omega,\omega_1}$ at distance 1 from each other, so $C(K^{\omega,\omega_1})\cong X^{\omega,\omega_1}$ cannot contain a countable dense set.

If $C(\mathbb{N}^*)$ had the BFPP, then by Theorem~\ref{retractN} and Lemma \ref{BFPPretract},  $C(K^{\omega,\omega_1})$ would also have it. However, since  $C(K^{\omega,\omega_1})$ is isometric to $X^{\omega,\omega_1}$, it follows from  Theorem~\ref{XnotBFPP} that $C(K^{\omega,\omega_1})$ does not have the BFPP.

\end{proof}

In general, for every locally compact and $\sigma$-compact topological space $Z$, the  remainder space $Z^* = \beta Z\setminus Z$
 is an $F$-space (see for instance \cite[Proposition 1.5.12]{Lau}). Thus,  it is natural to wonder about the BFPP for these kind of compact spaces. For example, what about $[0,+\infty)^*$? (Remember that $C(\beta[0,+\infty))$ fails the BFPP by Corollary \ref{sequence}. At first sight, the idea of relating the space $[0,+\infty)^*$ to the ones previously considered via retracts would seem hopeless, because $[0,+\infty)^*$ is now connected. However, we can still work at the level of Banach spaces using averaging operators. 

\begin{cor}
	Under CH, $C([0,+\infty)^*)$ fails the BFPP.
\end{cor}

\begin{proof}
 We identify $C(\mathbb{N}^*)$ with the quotient space $Y = \ell_\infty/c_0$ and we identify $C([0,+\infty)^*)$ with the quotient space $X$ of $C_b[0,\infty)$ of bounded continuous functions on $[0,\infty)$ by the subspace of functions with zero limit at infinity. Define an operator $E\colon X\To Y$ that sends the equivalence class of a function $f\colon [0,+\infty)\To \mathbb{R}$ to the equivalence class of the sequence $(f(n))_{n\in\mathbb{N}}$. Define an operator $R\colon Y\To X$ that sends the equivalence class of a sequence $(a_n)_{n\in\mathbb{N}}$ to the equivalence class of the function $$f(x) =  \begin{cases} a_x & \text{ if } x\in \mathbb{N},\\  (\lceil x \rceil - x)\cdot a_{\lfloor x\rfloor} + (x-\lfloor x\rfloor) \cdot a_{\lceil x\rceil } & \text{ if } x\not\in\mathbb{N}, \end{cases}$$
 where by $\lfloor x\rfloor$ and $\lceil x\rceil$  we denote the floor and ceiling function respectively; and $a_0:=0$.

 These operators satisfy the assumptions of Lemma~\ref{retractBanach}. Thus, if $C([0,+\infty)^*)$ had the BFPP so would $C(\mathbb{N}^*)$, which contradicts   Theorem~\ref{CH}.

\end{proof}

\section{Some questions and remarks}

In Section \ref{sectionExamples} we proved that the space $C(K^{\kappa,\Gamma},p^{\kappa,\Gamma})$, which is isometric to $X^{\kappa,\Gamma}_0$, fails the BFPP. Motivated by this family of examples, we can raise the following general question: Given an infinite compact set $K$ and a non-isolated $p\in K$,  is it always the case that $C(K,p):=\{f\in C(K): f(p)=0\}$ fails the BFPP? We do not have a complete answer to this question. However we can provide some answers to some particular cases.\\

Recall that $p\in K$ is  said to be a $P$-point  if every continuous function $g\in C(K)$ is constant on an open neighborhood of $p$. Equivalently, if a countable intersection of open neighborhoods of $p$ always contains an open neighborhood of $p$. If $p\in K$ fails to be a $P$-point, we say that $p$ is a non-$P$-point. In case that $K$ contains a nontrivial convergent sequence, its limit point is a non-$P$-point.  In fact, it is well-known that  every infinite compact space contains  non-$P$-points.
 Indeed, if all points of $K$ were $P$-points, by compactness all continuous functions on $K$  would be finite valued.
But if $\{x_n\}$ is a sequence of different points of $K$ with $x_n \notin \overline{\{x_m: m\neq n\}}$ for every $n\in \mathbb{N}$ and $f_n\colon K\To [0,1]$ are continuous functions with $f_n(x_m)=1$ if $m = n$ and $f_n(x_m)=0$ if $n \neq m$, then  $f:=\sum_{k=1}^\infty {1\over 2^k} f_k$ is a  continuous function that verifies  $f(x_n)={1\over 2^{n}}$ for all $n\in\mathbb{N}$,  producing infinitely many different values.  
\\

\begin{theorem}\label{C0}
If $p$ is a non-$P$-point of $K$, then  $C(K,p)$ fails the BFPP.
\end{theorem}

\begin{proof}

Consider $g\in C(K)$ such  that $g$ fails to be constant at any open neighborhood of $p$. 
 Without loss of generality we can assume that $g(p)=0$ and, considering $\min\{|g|,1\}$, we can suppose that $0\le g\le 1$. 
For all $f\in C(K)$ define $Tf=(1-g)f +g$. Notice that $T$ is nonexpansive since $\V Tf-Th\V_\infty\le \V f-h\V_\infty$ for all $f,h\in C(K)$ and $T(B_{C(K,p)})\subset B_{C(K,p)}$. Nevertheless, $T$ is fixed point free. Indeed,  if there were  some $f\in C(K,p)$ such that $Tf=f$ then $g(1-f)=0$. This would imply that 
the function  $g$ has to be null on the open set $f^{-1}({-1\over 2},{1\over 2})$, which  clearly contains $p$.
\end{proof}

Theorem \ref{C0} includes, as a particular case, the Banach space $c_0$. As a remark concerning the separable Banach spaces  $c_0$ and $c$, it is known that the only  closed convex bounded subsets in $c_0$  with the FPP  are the weakly compact ones \cite{DLT}, and  it is  an open problem to characterize which closed convex bounded subsets of $c$ verify the FPP, since they stretch out of the family of weakly compact ones \cite{BS,EJS, GLP}.\\

From  Theorem \ref{C0} we can also conclude that for every infinite compact space $K$ there is some point $p\in K$ such that $C(K,p)$  fails the BFPP.\\

It is natural to ask whether the converse of Theorem \ref{C0} holds whenever $C(K)$ has the BFPP, i.e.~ if $C(K,p)$ has the BFPP whenever $p$ is a $P$-point of $K$ and $C(K)$ has the BFPP. We do not know the answer to this question.

For the class of spaces $C(K^{\kappa,\Gamma},p^{\kappa,\Gamma})$  introduced in Section \ref{sectionExamples} and for which we proved the failure of the BFPP, the point $p^{\kappa,\Gamma}$  is in fact  a $P$-point of $K^{\kappa,\Gamma}$.
Indeed,  the property of being a $P$-point of $K$ can be translated into saying that for every $f\in C(K,p)$ there exists $g\in C(K,p)$ such that $f\cdot (1-g) =0$. For every $f\in X^{\kappa,\Gamma}_0$ there is $g\in X^{\kappa,\Gamma}_0$ with $f\cdot (1-g) = 0$. Just make sure that $g$ takes value 1 on the small set where $f$ does not vanish. 
Nevertheless, we have shown that the space $X^{\kappa,\Gamma}$ already fails the BFPP.

%
%
%
%

For the topological compact space $\beta\mathbb{N}$, the only $P$-points  are the principal ultrafilters, which are indeed isolated points of $\beta\mathbb{N}$. So $C(\beta\mathbb{N},p)$ has the BFPP if and only if $p$ is a principal ultrafilter, in which case $C(\beta\mathbb{N},p) \cong C(\beta\mathbb{N})\cong \ell_\infty$. The existence of $P$-points in the \v{C}ech-Stone remainder $\mathbb{N}^*$ is a much more delicate set-theoretic question. One can construct such $P$-points under the Continuum Hypothesis \cite{RudinPpoints}, but there are other models of set theory in which $\mathbb{N}^*$ contains no $P$-points \cite{ShelahPpoints}.\\

We finally conclude the paper by highlighting several  questions that may arise after reading the article, which we believe could be of significant interest for future research in the area:
\begin{itemize} 

\item[i)] Is it possible to find an infinite compact set $K$ and a non-isolated point $p\in K$ such that $C(K,p)=\{f\in C(K): f(p)=0\}$ has the BFPP?

\item[ii)] Is it possible to construct an explicit  example of a fixed point free nonexpansive mapping from the unit ball of $\ell_\infty/c_0$ into itself (without turning to CH)?

\item[iii)] Is it possible to exhibit an example of a fixed point free nonexpansive mapping from the unit ball of $C(K)$ into itself, whenever $C(K)$ fails to be order-complete? In this case, we would have a topological characterization of the BFPP  for Banach spaces of continuous functions:  $C(K)$ satisfies the BFPP if and only if $K$ is extremally disconnected.

\end{itemize}

\end{document}